\documentclass{amsart}
\usepackage{amssymb}
\usepackage[all]{xy}
\newtheorem{theorem}{Theorem}[section]
\newtheorem{lemma}[theorem]{Lemma}

\newtheorem{proposition}[theorem]{Proposition}
\newtheorem{definition}[theorem]{Definition}
\theoremstyle{definition}  

\newtheorem{example}[theorem]{Example}
\newcommand{\C}{\mathcal C}
\newcommand{\D}{\mathcal D}

\newcommand{\F}{\mathcal F}

\newcommand{\K}{\mathcal K}
\newcommand{\T}{\mathcal T}

\newcommand{\cO}{\mathcal O}
\newcommand{\BC}{\mathbb C}

\newcommand{\BZ}{\mathbb Z}
\newcommand{\BF}{\mathbb F}

\newcommand{\BQ}{\mathbb Q}

\newcommand{\FP}{\mbox{FPdim}}
\newcommand{\Rep}{\mbox{Rep}}
\newcommand{\Hom}{\mbox{Hom}}
\newcommand{\Aut}{\mbox{Aut}}
\newcommand{\id}{\mbox{id}}
\newcommand{\coeval}{\mbox{coeval}}
\newcommand{\eval}{\mbox{eval}}
\newcommand{\Id}{\mbox{Id}}
\newcommand{\Tr}{\mbox{Tr}}
\newcommand{\Ve}{\mbox{Vec}}

\def\tareesidedbox#1{\setbox0=\hbox{$#1$}\dimen0=\wd0 \advance\dimen0 by3pt\rlap{\hbox{\vrule height9pt width.4pt
 depth2pt \kern-.4pt\vrule height9.4pt width\dimen0 depth-9pt\kern-.4pt \vrule height9pt width.4pt depth2pt}}
 \relax \hbox to\dimen0{\hss$#1$\hss}}

\def\tareesidedboy#1{\setbox0=\hbox{$#1$}\dimen0=\wd0 \advance\dimen0 by3pt\rlap{\hbox{\vrule height9.8pt width.4pt
 depth2pt \kern-.4pt\vrule height10pt width\dimen0 depth-9.6pt\kern-.4pt \vrule height9.8pt width.4pt depth2pt}}
 \relax \hbox to\dimen0{\hss$#1$\hss}}

\begin{document}

\title{ON BRAIDED NEAR-GROUP CATEGORIES}
\author{JOSIAH THORNTON}

\maketitle

\begin{abstract}  We prove that any fusion category over $\BC$ with exactly one non-invertible simple object is spherical.  Furthermore, we classify all such categories that come equipped with a braiding.
\end{abstract}

\section{Introduction and statements of results}

This paper is devoted to the study of a type of fusion categories introduced by Jacob Siehler.  For this paper, we will assume that the ground field is $\BC.$  Similar results can be found for any algebraically closed field of characteristic $0.$

\begin{definition}\cite{S1} A near-group category is a semisimple, rigid tensor category with finitely many simple objects (up to isomorphism) such that all but one of the simple objects is invertible.  In the language of fusion categories, a near-group category is a fusion category with one non-invertible simple object.  If such a category comes equipped with a braiding, then we call it a braided near-group category.
\end{definition}

We show (see section 2) the well-known result that the Grothendieck ring of a near-group category is determined by a finite group $G$ and a non-negative integer $k.$  For each near-group category $\C,$ we call the data $(G,k)$ the near-group fusion rule of $\C.$\\

\begin{example} (i) Near-group categories with fusion rule $(G,0)$ for some finite group $G$ are known as Tambara-Yamagami categories.  These categories are classified up to tensor equivalence in \cite{TY}.  When they come equipped with a braiding, they are classified up to braided tensor equivalence in \cite{S1}.\\
(ii)  The well-known Yang Lee (see \cite{O1}) categories are precisely the near-group categories with fusion rule $(1,1).$  Up to tensor equivalence, there are two such categories, each of these admitting two braidings.\\
(iii) Let $\C$ be the fusion category associated to the affine $\mathfrak{sl}_2$ on level $10$ and let $A\in\C$ be the commutative $\C$-algebra of type $E_6$.  The category $\Rep(A)$ of right $A$-modules contains a fusion subcategory (see \cite[Section 4.5]{O2}) which is a near-group category with fusion rule $(\BZ/2\BZ,2).$\\
(iv) The Izumi-Xu category $\mathcal{IX}$ (see \cite[Appendix A.4]{CMS}) is a near-group category with fusion rule $(\BZ/3\BZ,3).$
\end{example}

J. Siehler has many results on near-group categories.  With \cite[Theorem 1.1]{S2} he proves that if $k\neq0,$ then $|G|\leq k+1.$  Siehler also classisfied braided near-group categories with near-group fusion rule $(G,0)$ in \cite[Theorem 1.2]{S1}.

Let $\C$ be a tensor category.  Recall (see \cite{DGNO}) that a spherical structure on $\C$ is an isomorphism of tensor functors $\varphi:\Id\to\ast\ast$ so that for every simple object $V\in\C,$ we have
$$\dim(V)=\Tr_V(\varphi)=\Tr_{V^\ast}(\varphi)=\dim(V^\ast).$$
Let $\varphi$ be an field automorphism of $\BC.$ Recall  that a spherical structure is called $\varphi$-pseudounitary if $\varphi(\dim(V))>0,$ for all simple objects $V.$

Our first main theorem, which is proved in section 2, is a result for (not necessarily braided) near-group categories.  This result is positive evidence for the question \cite{ENO} whether all fusion categories admit a spherical structure.

\begin{theorem} \label{sph} Any near-group category is spherical, moreover it is $\varphi$-pseudounitary for a suitable choice of $\varphi.$
\end{theorem}

From the results of P. Deligne and  G. Seitz, we are able to deduce the following classification of symmetric near-group categories.
\begin{proposition} \label{symmetric}
Let $\C$ be a symmetric near-group category with near-group fusion rule $(G,k)$ and $k\neq0.$  Then $\C$ is braided tensor equivalent to $\Rep(\BF_{p^l}\rtimes\BF_{p^l}^\ast),$ for some $p^l\neq2.$
\end{proposition}

Our main result is from the study of non-symmetric braided near-group categories. We prove our main theorem in section 4.

\begin{theorem} \label{class} Let $\C$ be a non-symmetric, braided, near-group category with fusion rule $(G,k)$ where $k\neq0,$ then $G$ is either the trivial group, $\BZ/2\BZ$ or $\BZ/3\BZ.$  Furthermore if $G$ is trivial, then there are are four associated braided near group categories (up to braided tensor equivalence).  All of these categories have fusion rule $(1,1).$  If $G=\BZ/2\BZ,$ then there are another two associated near-group categories, both with near-group fusion rule $(\BZ/2\BZ,1).$  And finally, if $G=\BZ/3\BZ,$ then there is one associated category with near-group fusion rule $(\BZ/3\BZ,2).$
\end{theorem}

We note that Proposition \ref{symmetric}, Theorem \ref{class} and J. Siehler's classification of braided near-group categories with fusion rule $(G,0)$ give a complete classification of braided near-group categories.

\subsection{Acknowledgments}  I would like to thank Victor Ostrik for teaching me everything I know about fusion categories.  His guidance and instruction made this paper possible.  I would also like to thank Eric Rowell for the suggestions and corrections he has made.

\section{Near-group categories are spherical}
The goal of this section is to prove Theorem \ref{sph}.

\subsection{Near-group fusion rule}

Let $\C$ be a near-group category with non-invertible object $X.$  The set of invertible objects of $\C$, denoted $\cO(\C),$  forms a group where multiplication is given by the tensor product structure on the category.  Therefore we can associate a finite group $G$ to any near-group category.  Let $g\in G$ represent an invertible object of $\C.$  Since $g$ is invertible and $X$ is not invertible, $g\otimes X$ is a non-invertible simple object of $\C,$ therefore $g\otimes X\simeq X$ for all $g\in G.$  Similarly, $X^\ast$ is a non-invertible simple object, and therefore $X^\ast\simeq X.$  Therefore
$$\Hom(g,X\otimes X)\cong \Hom(1,g^\ast\otimes X\otimes X)\cong\Hom(1,X\otimes X)\neq0.$$
Thus $g$ appears as a summand of $X\otimes X$ for each $g\in G.$  Since $\dim\Hom(1,X\otimes X)=1,$ $g$ appears as a summand of $X\otimes X$ exactly once.  Therefore we may decompose
$$X\otimes X\simeq\bigoplus_{g\in G} g\oplus kX$$
for some $k\in\BZ_{\geq0}.$ 

\subsection{Sphericalization of a near-group category}

For any fusion category, $\C,$ we are given $\gamma: Id\to \ast\ast\ast\ast$ an isomorphism of tensor functors by \cite[Theorem 2.6]{ENO}.  Then we may define the sphericalization of $\C.$

\begin{definition} \cite[Remark 3.1]{ENO}  The sphericalization, $\tilde{\C},$ of a fusion category $\C$ is the fusion category whose simple objects are pairs $(V,\alpha)$ where $V\in\cO(\C)$ and $\alpha:V\stackrel{\sim}{\to}V^{\ast\ast}$ satisfies $\alpha^{\ast\ast}\alpha=\gamma.$  This category has a canonical spherical structure $i:\Id\to\ast\ast.$
\end{definition}

Fix an isomorphism $f:V\to V^{\ast\ast}.$ Since $\Hom(V,V^{\ast\ast})$ is one dimensional, we may write $\alpha=a\cdot f$ for some $a\in\BC^\times.$  We also have $\alpha^{\ast\ast}=a\cdot f^{\ast\ast}.$  Similarly, we may write $\gamma=z\cdot f^{\ast\ast}f$ for some $z\in\BC^\times.$  Then the condition $\alpha^{\ast\ast}\alpha=\gamma$ is equivalent to $a^2=z.$  Therefore for each $V\in\cO(\C),$ we have two such $\alpha.$  Fixing one, we write $(V,\alpha)=V_+$ and $(V,-\alpha)=V_-.$

Now let $\C$ be a near-group category with non-invertible simple object $X$ and fusion rule $(G,k).$  For $X_+,X_-\in\tilde{\C},$ let $d=\dim(X_+)=\Tr_{X_+}(i)=\Tr_X(\alpha),$ so that $\dim(X_-)=\Tr_{X_-}(i)=\Tr_X(-\alpha)=-d.$ Similarly for $g\in G,$ define $g_+\in\tilde{\C}$ to be the simple object with $\dim(g_+)=1$ and whose image under the forgetful functor $\F:\tilde{\C}\to\C$ is $g.$

\subsection{Some technical lemmas}

\begin{lemma} For $e\in G$ the identity, $e_-\otimes X_+\simeq X_-.$  Furthermore for $g,h\in G,$ we have $g_+\otimes X_\pm\simeq X_\pm$ and $g_+\otimes h_+\simeq (g\otimes h)_+.$
\end{lemma}

\begin{proof}  Applying the forgetful functor $\F(e_-\otimes X_+)\simeq\F(e_-)\otimes \F(X_+)\simeq X,$ and $\dim(e_-\otimes X_+)=-\dim(X_+),$ so $e_-\otimes X_+\simeq X_-.$\\
Similarly we have $\F(g_+\otimes X_+)\simeq\F(g_+)\otimes\F(X_+)\simeq g\otimes X\simeq X,$  and $\dim(g_+\otimes X_+)=\dim(g_+)\dim(X_+)=1\cdot d=d.$  Therefore $g_+\otimes X_+\simeq X_+.$  An analogous proof shows $g_+\otimes X_-\simeq X_-.$\\
Finally $\F(g_+\otimes h_+)\simeq g\otimes h,$ and $\dim(g_+\otimes h_+)=1,$ so $g_+\otimes h_+\simeq (g\otimes h)_+.$

\end{proof}

\begin{lemma}  Let $\C$ be a near-group category with fusion rule $(G,k)$ and non-invertible object $X,$ and let $\tilde{\C}$ be the sphericalization of $\C.$  We have $(X_\pm)^\ast\simeq X_\pm,$ and furthermore 
$$X_+\otimes X_+\simeq X_-\otimes X_-\simeq \bigoplus_{g\in G}g_+\oplus sX_+\oplus tX_-,$$
where $s+t=k.$
\end{lemma}

\begin{proof}  Clearly the forgetful functor maps $(X_+)^\ast\mapsto X^\ast,$ therefore $(X_+)^\ast\simeq X_+$ or $X_-.$  Since $\dim(X_-)=-\dim(X_+),$ we conclude that $(X_+)^\ast\simeq X_+.$\\
Since $(X_+)^\ast\simeq X_+,$ we have for each $g\in G,$
$$\Hom(g_+,X_+\otimes X_+)\cong\Hom(1,(g^{-1})_+\otimes X_+\otimes X_+)\cong\Hom(1,X_+\otimes X_+)\neq0.$$
Therefore $g_+$ appears as a summand of $X_+\otimes X_+$ for each $g\in G.$  By applying the forgetful functor, we see that $g_+$ appears as a summand at most once.  This gives us
$$X_+\otimes X_+\simeq \bigoplus_{g\in G}g_+\oplus sX_+\oplus tX_-,$$
with no restriction on $s,t.$  Again applying the forgetful functor gives 
$$X\otimes X\simeq \bigoplus_{g\in G} g\oplus (s+t)X,$$
 and the lemma is proved after noting
 $$X_-\otimes X_-\simeq (e_-\otimes X_+)\otimes (e_-\otimes X_+)\simeq X_+\otimes X_+.$$
\end{proof}

After renaming of $X_+,$ we may assume $s-t\geq0.$

\begin{lemma} For $X_+\in\tilde{C},$ we have $d=\dim(X_+)=\frac{r\pm\sqrt{r^2+4n}}{2}$ and $\dim(\C)=\frac{r^2+4n\pm r\sqrt{r^2+4n}}{2}$ where $n=|G|$ and $r=s-t\geq0.$
\end{lemma}

\begin{proof}
We have
$$d^2=\dim(X_+\otimes X_+)=\dim\left(\bigoplus_{g\in G}\oplus sX_+\oplus tX_-\right)=|G|+(s-t)d=n+rd.$$
And
$$\dim(\C)=|G|+d^2=2n+rd=\dfrac{r^2+4n\pm r\sqrt{r^2+4n}}{2}.$$
\end{proof}

\begin{lemma} \label{r} Recall $r=s-t.$\\
(a) If $\tilde{\C}$ is pseudo-unitary, then $r=k,$\\
(b) If $\sqrt{r^2+4n}\in\BZ,$ then $r=k,$\\
(c) If $\sqrt{k^2+4n}\in\BZ,$ then $r=k.$
\end{lemma}

\begin{proof}
(a) If $\tilde{\C}$ is pseudo-unitary, then $\dim(\tilde{\C})=\FP(\tilde{\C})=2\FP(\C).$  Therefore $r^2+4n+r\sqrt{r^2+4n}=k^2+4n+k\sqrt{k^2+4n},$ and since $|r|=|s-t|\leq k,$ we have $r=k.$

(b) If $\sqrt{r^2+4n}\in\BZ,$ then $d$ is a rational algebraic integer, therefore $d\in \BZ.$  By \cite[Lemma A.1]{HR} $\tilde{\C}$ is pseudo-unitary, and $r=k$ by (a).\\

(c)  If $\sqrt{k^2+4n}\in\BZ,$ then $\FP(X)\in\BZ.$ By \cite[Proposition 8.24]{ENO} $\tilde{\C}$ is pseudo-unitary and $r=k$ by $(a).$

\end{proof}

We will also use the following well-known lemma about algebraic integers.

\begin{lemma} \label{alg} Let $a,b,c,d\in\BZ$ such that $\sqrt{b},\sqrt{d}\not\in\BZ.$  Then $\frac{a+\sqrt{b}}{c+\sqrt{d}}$ is an algebraic integer if and only if $\frac{a-\sqrt{b}}{c-\sqrt{d}}$ is an algebraic integer.
\end{lemma}
\begin{proof}
Since $b\in\BZ$ is not a square, we may write $b=m\cdot p_1^{\beta_1}\cdots p_k^{\beta_k}$ for some square $m\in\BZ$ and primes $p_1,\ldots,p_k$ and odd integers $\beta_1,\ldots,\beta_k.$  Similarly, we may write $d=n\cdot q_1^{\delta_1}\cdots q_l^{\delta_l}$ for some square $n,$ primes $q_1,\ldots, q_l$ and odd integers $\delta_1,\ldots, \delta_l.$  We will consider two cases.\\
Case (i): Up to ordering $p_1=q_1.$ Let $\sigma\in Gal(\BQ(\sqrt{b},\sqrt{d})/\BQ)$ be the element which maps $\sqrt{p_1}$ to $-\sqrt{p_1}$ and fixes $\sqrt{p_i}$ for $i\neq1$ and $\sqrt{q_j}$ for $j\neq1.$  Then $\sigma(a+\sqrt{b})=a-\sqrt{b},$ $\sigma(c+\sqrt{d})=c-\sqrt{d}$ and $\sigma\left(\frac{a+\sqrt{b}}{c+\sqrt{d}}\right)=\frac{a-\sqrt{b}}{c-\sqrt{d}}.$\\
Case (ii): $p_i\neq q_j$ for all $1\leq i\leq k,$ $1\leq j\leq l.$  Then let $\sigma\in Gal(\BQ(\sqrt{b},\sqrt{d})/BQ)$ be the element which maps $\sqrt{p_1}$ to $-\sqrt{p_1}$, maps $\sqrt{q_1}$ to $-\sqrt{q_1}$ and fixes $\sqrt{p_i}$ for $i\neq1$ and $\sqrt{q_j}$ for $j\neq1.$  Then as above
$\sigma\left(\frac{a+\sqrt{b}}{c+\sqrt{d}}\right)=\frac{a-\sqrt{b}}{c-\sqrt{d}}.$

\end{proof}

\subsection{Proof of Theorem \ref{sph}}
Let $D=\dim(\tilde{\C})=r^2+4n+r\sqrt{r^2+4n},$ and $\Delta=\FP(\tilde{\C})=k^2+4n+k\sqrt{k^2+4n}.$  Then by \cite[Proposition 8.22]{ENO}
$$\dfrac{D}{\Delta}=\dfrac{r^2+4n+r\sqrt{r^2+4n}}{k^2+4n+k\sqrt{k^2+4n}}$$
is an algebraic integer.   Our goal is to prove that $r=k,$ thus proving the theorem.\\
When either $\sqrt{r^2+4n}$ or $\sqrt{k^2+4n}$ are integers, we know $r=k$ by Lemma \ref{r}.  Therefore assume $\sqrt{r^2+4n},\sqrt{k^2+4n}\not\in\BZ.$  Then
$$\dfrac{r^2+4n-r\sqrt{r^2+4n}}{k^2+4n-k\sqrt{k^2+4n}}$$
is an algebraic integer by Lemma \ref{alg}, and thus
$$\left(\dfrac{r^2+4n+r\sqrt{r^2+4n}}{k^2+4n+k\sqrt{k^2+4n}}\right)\left(\dfrac{r^2+4n-r\sqrt{r^2+4n}}{k^2+4n-k\sqrt{k^2+4n}}\right)=\dfrac{4n(r^2+4n)}{4n(k^2+4n)}=\dfrac{r^2+4n}{k^2+4n}$$
is an algebraic integer.  Therefore $r^2=k^2,$ and $r=k,$ since $r\geq0.$

The full tensor category generated by simple objects $\{g_+\}_{g\in G}\cup\{X_+\}$ is tensor equivalent (by the forgetful functor) to $\C.$   Therefore $\C$ is tensor equivalent to a full tensor subcategory of a spherical category and therefore spherical itself.  Moreover if $d>0,$ then $\C$ is pseudo-unitary, and if $d<0,$ then $\C$ is $\varphi$-pseudounitary. $\hfill \square$\\

\subsection{Near-group categories with integer Frobenius-Perron dimension.}
\begin{proposition} \label{pseudo}  If a near-group category $\C$ with near-group fusion rule $(G,k)$ has integer Frobenius-Perron dimension, then either $k=0$ or $k=|G|-1.$  In the latter case $\FP(\C)=|G|(|G|+1).$
\end{proposition}
\begin{proof}
$\FP(X)=\frac{1}{2}(k+\sqrt{k^2+4n}).$  Therefore if $\FP(\C)\in\BZ,$ then $\FP(X)^2=\frac{1}{2}(k^2+2n+k\sqrt{k^2+4n})$ is an integer, and $\sqrt{k^2+4n}\in\BZ.$  Therefore $k^2+4n=(k+l)^2$ for some $\l\in\BZ_{>0}.$  Expanding, we get $4n=2kl+l^2.$  Therefore $l$ is even and $l=2p$ for $p\in\BZ_{>0}.$ Finally, $k+1\leq kp+p^2\leq n\leq k+1$ by \cite[Theorem 1.1]{S2} when $k\neq0.$

Therefore $k=0$ or $k=n-1=|G|-1.$  In the latter case, $\FP(X)=k+1=|G|,$ and $\FP(\C)=n+(k+1)^2=n+n^2=|G|(|G|+1).$
\end{proof}

\section{M\"uger center of a braided near-group category}

The goal of this section is to prove Proposition \ref{symmetric}.  We will use mostly definitions and results from \cite{DGNO}.

Let $\C$ be a braided tensor category. From \cite{Mu2}, we define the M\"uger center of $\C$ to be the full tensor subcategory of $\C$ with objects 
$$\{X\in \C|\sigma_{Y,X}\sigma_{X,Y}=\id_{X\otimes Y}\ \forall Y\in\C\}.$$
  We denote the M\"uger center of $\C$ by $\C'.$

\subsection{The M\"uger center of a near-group category contains all invertible objects.}
Recall the following definitions for braided fusion categories.

\begin{definition} \cite[Section 2.2; Section 3.3]{DGNO}
Let $\C$ be a fusion category:\\
(a) Define $\C_{ad}$ to be the fusion subcategory generated by $Y\otimes Y^\ast$ for $Y\in\cO(\C).$\\
(b) For $\K$ a fusion subcategory of $\C,$ we define the \emph{commutator of $\K$} to be the fusion subcategory $\K^{co}\subseteq\C,$ generated by all simple objects $Y\in\cO(\C),$ where $Y\otimes Y^\ast\in\cO(\K).$
\end{definition}

Let $\C$ be a near-group category with near-group fusion rule $(G,k).$  Recall that for this proposition we assume $k\neq0.$

\begin{lemma} \label{adjoint} $\C_{ad}=\C.$
\end{lemma}

\begin{proof} This is clear as $X\simeq X^\ast$ and $X\otimes X\simeq G\oplus kX,$ thus contains all simple objects of $\C$ as summands.
\end{proof}

\begin{lemma}\cite[Proposition 3.25]{DGNO} \label{DGNO 3.25} Let $\K$ be a fusion subcategory of a braided fusion category $\C.$  Then $(\K_{ad})'=(\K')^{co}.$
\end{lemma}

Letting $\K=\C$ in Lemma \ref{DGNO 3.25}, we get

\begin{proposition} \label{center}  Let $\C$ be a braided near-group category with fusion rule $(G,k).$  If $k>0,$ then the M\"uger center $\C'\subseteq\C$ is either $\C$ or $\Ve_G.$
\end{proposition}

\begin{proof}
$\C'=(\C_{ad})'=(\C')^{co}\supset G.$
\end{proof}

In particular, $\C$ can only be modular if $G$ is trivial.

\subsection{Symmetric tensor categories}

Let $A$ be a group.  Deligne \cite{D} defines $\Rep(A,z)$ to be the category of finite dimensional super representations $(V,\rho)$ of $A,$ where $\rho(z)$ is the automorphism of parity of $V.$  In \cite{DGNO} this is presented as the fusion category $\Rep(G)$ with $z\in Z(G)$ satisfying $z^2=1$ and braiding $\sigma'$ given by
$$\sigma'_{UV}(u\otimes v)=(-1)^{ij}v\otimes u\ \mathrm{if}\ u\in U, v\in V, zu=(-1)^iu, zv=(-1)^jv.$$
In \cite[Corollaire 0.8]{D} it is shown that any symmetric fusion category is equivalent to $\Rep(A,z)$ for some choice of finite group $A,$ and central element $z\in A$ with $z^2=1.$  If $z\neq1$ we call such a category super-Tannakian.  If $z=1,$ then $\Rep(A,z)=\Rep(A)$ and it is called Tannakian.  Note that $\Rep(A/\langle z\rangle)$ is the subcategory of modules $M$ where $z$ acts trivially on $M.$  This is a maximal Tannakian subcategory of $\Rep(A,z).$

Recall (see \cite[Example 2.42]{DGNO}) $s\Ve$ is defined to be the category $\Rep(\BZ/2\BZ,z),$ where $z$ is the non-trivial element of $\BZ/2\BZ.$  The following lemma is due to \cite[Lemma 5.4]{Mu1} and \cite[Lemma 3.28]{DGNO}.  This lemma will be used  to show that particular categories do not exist.

\begin{lemma} \label{twist} Let $\C$ be a braided fusion category and $\delta\in \C'$ an invertible object such that the fusion subcategory of $\C$ generated by $\delta$ is braided equivalent to $s\Ve.$  Then for all $V\in\cO(\C),$  $\delta\otimes V$ cannot be mapped to $V$ by some tensor automorphism.
\end{lemma}

\begin{proof}
For $V\in\cO(\C),$ let $\mu_V$ be defined to be the compostion
$$V\xrightarrow{\id_V\otimes\coeval_{V^\ast}}V\otimes V^\ast\otimes V^{\ast\ast}\xrightarrow{\sigma_{V,V^\ast}\otimes\id_{V^{\ast\ast}}}V^\ast\otimes V\otimes V^{\ast\ast}\xrightarrow{\eval_V\otimes_{V^{\ast\ast}}}V^{\ast\ast}$$
where $\sigma$ is the braiding on $\C.$  It is well known (see \cite[Lemma 2.2.2]{BK}) that for $V,W\in\cO(\C),$ $\mu_V$ and $\mu_W$ satisfy
$$\mu_V\otimes\mu_W=\mu_{V\otimes W}\sigma_{W,V}\sigma_{V,W}.$$
 Therefore since, $\delta\in\C',$ we have
$$\mu_{\delta\otimes V}=\mu_{\delta}\otimes\mu_V,$$
for all $V\in\cO(\C).$ Since $\delta$ generates $s\Ve,$ we know that $\sigma'(\delta,\delta)=-\id_1$ and $\mu_\delta=-\id_\delta.$  Recall \cite{ENO} for a simple object $U\in\cO(\C),$ we define $d_+(U)$ to be the composition
$$1\xrightarrow{\coeval_V}V\otimes V^\ast\xrightarrow{\mu\otimes \id_{V^\ast}}X^{\ast\ast}\otimes V^\ast\xrightarrow{\eval_{V^\ast}}1.$$
Then
\begin{eqnarray*}
d_+(\delta\otimes V)&=&\eval_{(\delta\otimes V)^\ast}\circ(\mu_{\delta\otimes V}\otimes\id_1)\circ\coeval_{\delta\otimes V}\\
&=&\eval_{\delta^\ast}\circ(-\id_\delta\otimes \id_1)\circ\coeval_\delta\cdot\eval_{V^\ast}\circ(\mu_V\otimes \id_1)\coeval_V\\
&=&-d_+(V).
\end{eqnarray*}
Since $d_+(V)\neq0,$ we have $d_+(V)\neq d_+(\delta\otimes V),$ and $V$  cannot be mapped to $\delta\otimes V$ by some automorphism.

\end{proof}

\subsection{Proof of Proposition \ref{symmetric}}

Let $\C$ be a symmetric near-group category with near-group fusion rule $(G,k).$ by \cite[Corollaire 0.8]{D} $\C$ is equivalent (as a tensor category) to $\Rep(H)$ for some finite group $H.$  Since $\C$ is a near-group category, $H$ has exactly one irreducible representation of dimension greater than one.  The following lemma classifies such groups. 
\begin{lemma}\cite{Se} \label{seitz} A group $G$ has exactly one irreducible $\C$-representation of degree greater than one if and only if (i) $|G|=2^k,$ $k$ is odd, $[G,G]=Z(G),$ and $|[G,G]|=2,$ or (ii) $G$ is isomorphic to the group of all transformations $x\mapsto ax+b,$ $a\neq0,$  on a field of order $p^n\neq2.$
\end{lemma}

By \cite[Corollaire 0.8]{D} and Lemma \ref{seitz}, $\C$ is tensor equivalent to $\Rep(H)$ where $|H|=2^l,$ or $H$ is isomorphic to the group of all transformations $x\mapsto ax+b,$ $a\neq0,$ on a field of order $p^l\neq2.$\\
If $|G|=2^l,$  then by Lemma \ref{pseudo}, $\Rep(G)$ is Tambara-Yamagami if it is near-group.  Therefore we may assume that $H$ is the latter group described above.  Such a group $H$ is isomorphic to $\BF_{p^l}\rtimes\BF_{p^l}^\ast$ since there is a split short exact sequence
$$1\to\BF_{p^l}\to H\to\BF_{p^l}^\ast\to1.$$
Therefore $\C$ is tensor equivalent to $\Rep(\BF_{p^l}\rtimes\BF_{p^l}^\ast).$  Since $Z(H)=1,$ there does not exist a braiding on $\Rep(H)$ making it a super-Tannakian category.  Therefore $\C$ is braided tensor equivalent to $\Rep(\BF_{p^l}\rtimes\BF_{p^l}^\ast).$ $\hfill \square$

\subsection{Equivariantization of a braided tensor category}

Let $\C$ be a braided tensor category.  
\begin{definition} \cite[Section 4.2]{DGNO} (i) Let $\underline{\Aut}^{\mathrm{br}}(\C)$ be the category whose objects are braided tensor equivalences and whose morphisms are isomorphism of braided tensor functors.\\
(ii) For a group G, let $\underline{G}$ be the tensor category whose objects are elements of $G,$ whose morphisms are the identity morphisms and whose tensor product is given by group multiplication.\\
(iii) We say that \emph{$G$ acts on $\C$ viewed as a braided tensor category} if there is a monoidal functor $\underline{G}\to\underline{\Aut}^{\mathrm{br}}(\C).$\\
(iv) We say $\C$ is a \emph{braided tensor category $\C$ over $\mathcal{E}$} if it is equipped with a braided functor $\mathcal{E}\to\C'.$
\end{definition}

Let $G$ be a group, and $G$ act on $\C$ viewed as a braided tensor category.  Then define the equivariantization of $\C$ by $G.$
\begin{definition}\cite[Definition 4.2.2]{DGNO}
Let $\C^G$ be the category with objects $G$-equivariant objects.  That is an object $X\in\C$ along with an isomorphism $\mu_g:F_g(X)\to X$ such that the following diagram commutes for all $g,h\in G.$
$$\xymatrix{ F_g(F_h(X)) \ar[rr]^{F_g(u_h)} \ar[d]_{\gamma_{g,h}(X)} && F_g(X) \ar[d]^{u_g}\\
F_{gh}(X) \ar[rr]^{u_{gh}} && X}$$
The morphisms in $\C^G$ are morphisms in $\C$ which commute with $u_g.$  The tensor product on $\C^G$ is the obvious one induced by the tensor product on $\C.$  Since the action of $G$ on $\C$ respects the braiding, there is an induced braiding on $\C^G.$
\end{definition}

The following propostition from \cite{DGNO} relates actions of $G$ on $\C$ and equivariantization.

\begin{proposition} \cite[Theorem 4.18(ii)]{DGNO} \label{DGNO 4.18}  Let $G$ be a finite group and $\C$ be a braided tensor category over $\Rep(G).$  Then there is a braided tensor category $\D$ equipped with an action of $G,$ such that $\D^G\cong\C.$
\end{proposition}

\subsection{Tannakian centers of braided near-group categories.}

Let $\C$ be a braided near-group category with near-group fusion rule $(G,k).$  Assume that $\C$ is not symmetric, so $\C'=\Ve_G.$  Therefore $\C'=\Rep(A,z)$ for some choice of finite group $A$ and $z\in A.$  For the remainder of this section, we will assume that $z\neq1$ and derive a contradiction.

Recall $H:=\Rep(A/\langle z\rangle)\subseteq \C'$ is a maximal Tannakian subcategory of $\Rep(A,z).$ By Proposition \ref{DGNO 4.18}, there exists a braided fusion category $\D$ and an action of $H$ on $\D$ so that $\D^H=\C$ and $\Ve^H=\Rep(H).$ Since $\C'$ is a braided tensor category over $\Rep(H),$ there also exists a category $\D_1\subset\D$ such that $\D_1^H=\C'.$  By \cite[Proposition 4.26]{DGNO}, we have $\FP(\D_1)=\FP(\C')/|H|=2.$  Let $\cO(\D_1)=\{1,Z\},$ and $\cO_1,\ldots,\cO_m$ denote the orbits of the simple objects of $\D$ under the action of $H,$ where $\cO_1=\{1\}$ and $\cO_2=\{Z\}.$  Since $\D_1^H=\C'$ and there is only one simple object of $\C$ not contained in $\C',$ there are only three orbits. The following proposition shows that no such category can exist.

\begin{proposition} There are no super-Tannakian categories $\T$ with the following structure:\\
(i) $\cO(\T)=\{1,\delta,T_1,\ldots, T_q\}$ where the fusion subcategory of $\T$ generated by $\delta$ is braided equivalent to $s\Ve,$\\
(ii) An action of a group $A$ on $\T$ transitively permuting $\{T_1,\ldots, T_q\}.$
\end{proposition}

\begin{proof}
Since $\delta$ is invertible, $\delta\otimes T_1\simeq T_s$ for some $1\leq s\leq q.$  By Lemma \ref{twist} there is no automorphism mapping $\delta\otimes T_1$ to $T_s.$  This contradicts the assumption that $A$ acts transitively on $\{T_1,\ldots,T_q\}.$
\end{proof}

Therefore we proved the following proposition.
\begin{proposition} \label{tanak} If $\C$ is a non-symmetric, braided, near-group category and $k\neq0,$ then the M\"uger center $\C'=\Rep(H)$ for some abelian group $H.$
\end{proposition}

\section{Classification of non-symmetric braided near-group categories.}
The goal of this section is to show there are $7$ non-symmetric, braided, near-group categories (up to braided tensor equivalence) which are not Tambara-Yamagami.  Again, we only care about the case when $k\neq0,$ as J. Siehler already did the classification when $k=0$ \cite{S1}. \\
In the previous section, we proved $\C'=\Rep(H).$  By Theorem \ref{DGNO 4.18}, there exists a braided fusion category $\D$ and an action of $H$ on $\D$ so that $\D^H=\C$ and $\Ve^H=G.$  Let $\cO_1,\ldots,\cO_m$ denote the orbits of the simple objects of $\D$ under the action of $H,$ where $\cO_1=\{1\}.$  Since $\Ve^H=G,$ we have $m=2,$ otherwise $G\cup\{X\}\subsetneq\cO(\C).$  For the remainder of this section, let $\cO_2=\{D_1,\ldots, D_s\}.$

\begin{lemma}  \label{dims1} If $s>1,$ then $\dim(D_j)=1$ for $1\leq j\leq s.$
\end{lemma}

\begin{proof} If $s\geq 2,$ then there exists $1\leq i\leq s,$ such that $D_i^\ast\not\simeq D_1.$  Therefore $D_1\otimes D_i=\bigoplus_{j=1}^s a_j D_j,$ and $d=\dim(D_j)$ satisfies the identity $d^2=d\sum_{j=1}^s a_j.$  This gives $d\in\BZ,$ so by the proof of \cite[Proposition 8.22]{ENO}, $d$ divides $\FP(D).$  Therefore, since $\FP(D)=1+sd^2,$ we have $d=1.$ 
\end{proof}

\begin{lemma} \label{orb1} If $\cO_2=\{D_1,\ldots, D_s\},$ then $s$ is either $1$ or $n.$
\end{lemma}

\begin{proof} By Lemma \ref{dims1}, if $s\geq2,$ then $\FP(\D)=1+s.$  Therefore $\FP(\C)=n(1+s)\in \BZ$ and $1+s=\FP(\D)=\frac{1}{n}\FP(\C)=n+1,$ since $\FP(\C)=n(n+1)$ by Proposition \ref{pseudo}.
\end{proof}

\begin{proposition} \label{yanglee}
If $\cO(\D)=\{1,D_1\},$ then either $\C=\D^H$ is a Yang-Lee category or $\C$ is Tambara-Yamagami.  Moreover there are (up to braided equivalence) four braided near-group categories $\C$ which are not Tambara-Yamagami and $\C_H$ is of rank two.
\end{proposition}
\begin{proof}
Let $D=D_1.$ Assume $D\otimes D=1.$ Let $X$ be the non-invertible object of $\C.$  Therefore $X$ is an equivariant object under the action of $H$ on $\C.$  Therefore $X=mD$ for some integer $m.$  Therefore $X\otimes X=m^21$ in $\C$ and must therefore lie in $\Rep(H)$ in $\C.$  In this case $\C$ is Tambara-Yamagami.\\

Now assume $D\otimes D=1\oplus D.$  Therefore $H$ acts on $\D$ trivially.  This gives $\C=\D\boxtimes\Rep(H),$ which is only a near-group category when $H$ is trivial and $\C=\D.$

The last part of the proposition is simply a note that there are are four Yang-Lee categories up to braided equivalence \cite{O1}.
\end{proof}

Since we just classified the case when $\D$ is of rank two, we will assume for the remainder of this section that $s>1,$ and therefore by Lemma \ref{dims1}, $\D$ is a pointed braided category which is non-degenerate by \cite[Corollary 4.30]{DGNO} since $\D^H=\C,$ where $\C'=\Rep(H).$  It is shown (see \cite{DGNO} or \cite{JS}) that a non-degenerate pointed braided category is classified by an abelian group $A$ and a non-degenerate quadratic form $q:A\to\BC^\times$ on $A.$  Note that $A$ is the group of isomorphism classes of simple objects. They denote such a category by $\C(A,q).$  Recall that the data $(A,q)$ for a finite abelian group $A$ and a non-degenerate quadratic form $q:A\to \BC^\times$ is called a metric group.

\begin{proposition} If $\D$ is of rank at least three, then $A$ is either $\BZ/2\BZ\oplus\BZ/2\BZ$ or $\BZ/3\BZ.$  Moreover if:\\
(i) $A=\BZ/2\BZ\oplus\BZ/2\BZ,$ then $q(a)=-1$ for every non-trivial element of $A,$\\
(ii) $A=\BZ/3\BZ,$ then $q(a)=q(b)$ are primitive third roots of unity for both non-trivial elements $a,b$ of $A.$
\end{proposition}

\begin{proof}
$\cO(\D)=\{1,D_1,\ldots,D_p\}$ where $H$ acts transitively on $\{D_1,\ldots, D_s\}$ by braided-tensor functors.  Let $A=\{e,d_1,\ldots,d_s\},$ then $o(d_i)=o(d_j)$ for $1\leq i,j\leq s.$  Therefore $A$ is an elementary abelian group.  Let $p=o(d_1).$  The action of $H$ on $\D$ gives rise to an action of $H$ on the metric group $(A,q)$ by morphisms $\{\varphi_h\}_{h\in H}$ of metric groups.  Since $H$ acts transitively on $\cO_2,$ we have for any $1\leq i,j\leq s$ we have $h\in H$ so that $\varphi_h(d_i)=d_j.$  Since $\varphi_h$ is a morphism of metric groups, we have $q(d_j)=q(\varphi_h(d_i))=q(d_i).$ Therefore it makes sense to define $\omega=q(d_i).$ By \cite[Remark 2.37 (i)]{DGNO}, we have $1=q(e)=q(d_1^p)=\omega^{p^2},$ therefore $\omega$ is a root of unity.

\cite[Corollary 6.3]{DGNO} states that for $(A,q)$ a metric group, we have
$$\Big|\sum_{a\in A}q(a)\Big|^2=|A|.$$
Therefore if $|A|=m,$ we have 
$$m=|1+(m-1)\omega|^2\geq(|(m-1)\omega|-1)^2=(m-2)^2=m^2-4m+4.$$
This gives $(m-1)(m-4)\leq0,$ so $m=2,3$ or $4$ and we assume $m\geq3.$  Since $A$ is an elementary abelian group, we know that $m=4$ implies $A=\BZ/2\BZ\oplus\BZ/2\BZ.$

Assume $A=\BZ/2\BZ\oplus\BZ/2\BZ.$  Then $|1+3\omega|^2=4,$ giving $\omega=-1.$

Finally, for $A=\BZ/3\BZ,$ we have $\omega=q(d_1)=q(d_1^2)=\omega^4,$ so $\omega$ is a third root of unity.  This gives $|1+2\omega|^2=3,$ and $\omega$ is a primitive third root of unity.
\end{proof}

\begin{proposition} \label{classify}
Let $\C$ be a non-symmetric braided near group category with fusion rule  $(G,k)$ where $k\neq0.$  There are two such categories (up to braided tensor equivalence) when $G=\BZ/2\BZ$ and one such category (up to braided tensor equivalence) when $G=\BZ/3\BZ.$
\end{proposition}
\begin{proof}
Assume $G=\BZ/2\BZ.$  We have shown above that $\C=\C(\BZ/3\BZ,q)^H,$ where $\Ve_{\BZ/2\BZ}=\Rep(H)$ (therefore $H=\BZ/2\BZ$) and $q:\BZ/3\BZ\to \BC^\times$ is defined by $q(a)=q(b)$ is a primitive third root of unity for non-trivial elements $a,b\in\BZ/3\BZ.$  For each of the two choices of $q,$ we have one non-trivial action of $H$ on $\C(\BZ/3\BZ,q).$  Therefore there are two non-symmetric near-group categories with fusion rule $(\BZ/2\BZ,1).$

Assume $G=\BZ/3\BZ.$  Then we showed that $\C=\C(\BZ/2\BZ\oplus\BZ/2\BZ,q)^H,$ where $\Ve_{\BZ/3\BZ}=\Rep(H)$ (therefore $H=\BZ/3\BZ$) and $q:\BZ/2\BZ\oplus\BZ/2\BZ\to\BC^\times$ is defined by $q(a)=q(b)=-1$ for the generators of $\BZ/2\BZ\oplus\BZ/2\BZ.$  Again, we only have one non-trivial action of $\BZ/3\BZ$ on $\C(\BZ/2\BZ\oplus\BZ/2\BZ,q).$   Therefore we get one non-symmetric near-group category with fusion rule $(\BZ/3\BZ,2).$
\end{proof}

 \bibliographystyle{ams-alpha}

\begin{thebibliography}{A} 

\bibitem[BK]{BK} B.~Bakalov, A.~Kirillov Jr.,
\textit{Lectures on Tensor categories and modular functors}, 
AMS, 2001.

\bibitem[CMS]{CMS} F.~Calegari, S.~Morrison, N.~Snyder,
\textit{Cyclotomic integers, fusion categories, and subfactors}, with an appendix by Victor Ostrik.  To appear in Communications in Mathematical Physics.  arXiv:1004.0665

\bibitem[D]{D} P. ~Deligne,
\textit{Categories Tensorielles}, Moscow Math. Journal 
\textbf{2} (2002) no. 2, 227-248

\bibitem[DGNO]{DGNO} V.~Drinfeld, S.~Gelaki, D.~Nikshych,
and V.~Ostrik,
\textit{On braided fusion categories I}, Selecta Math. (N.S.) \textbf{2} (2010), no. 1, 1-119.

\bibitem[ENO]{ENO} P.~Etingof, D.~Nikshych, V.~Ostrik,
\textit{On fusion categories}, Annals of Mathematics 
\textbf{162} (2005), 581-642.

\bibitem[HR]{HR} S.~Hong, E.~Rowell
\textit{On the classification of Grothendieck rings of non-self-dual modular categories}, Journal of Algebra
\textbf{132} (2010), 1000-1015.

\bibitem[JS]{JS} A.~Joyal, R.~Street
\textit{Braided tensor categories}, Adv. Math. \textbf{102} (1993), no. 1, 20-78.


\bibitem[Mu1]{Mu1} M.~M\"uger
\textit{Galois theory for braided tensor categories and modular closure}, Adv. Math. 
\textbf{150} (2000), no. 2, 151-200.

\bibitem[Mu2]{Mu2} M.~M\"uger
\textit{On the Structure of Modular Tensor Categories}, Proc. London Math. Soc. (3) 
\textbf{87} (2003), no. 2, 291-308.


\bibitem[O1]{O1} V.~Ostrik
\textit{Fusion Categories of Rank $2$}, Math. Res. Lett.
\textbf{10} (2003), 177-183.

\bibitem[O2]{O2} V.~Ostrik
\textit{Pre-modular categories of rank 3}, Mosc. Math. J.
\textbf{8} (2008), no. 1, 111-118.

\bibitem[Si1]{S1} J.~Siehler
\textit{Braided near-group categories}, math.qa/0111171.

\bibitem[Si2]{S2} J.~Siehler
\textit{Near-group categories}, , Algebr. Geom. Topol. 
\textbf{3} (2003), 719-775

\bibitem[Se]{Se} G.~Seitz
\textit{Finite groups having only one irreducible representation of degree greater than one.}, Proc. Amer. Math. Soc. \textbf{19} (1968), 459-461

\bibitem[TY]{TY} D.~Tambara, S.~Yamagami
\textit{Tensor categories with fusion rules of self-duality for finite abelian groups}, Journal of Algebra \textbf{209} (1998), 692-707.


\end{thebibliography}

\end{document}